\documentclass[10pt]{article}
\usepackage[ansinew]{inputenc}
\usepackage{booktabs,comment}

\usepackage{latexsym,bezier,enumerate,longtable,dcolumn,color,pstricks,xspace,curves,mathpazo,url}
\usepackage{amsmath,amsthm,amsopn,amstext,amscd,amsfonts,amssymb,fancybox,pst-node,pst-tree}

\usepackage{epsfig,epic,graphicx,pstricks}
\usepackage{tikz}

\usepackage{subcaption}
\usetikzlibrary{positioning}

\usepackage{algorithm} 
\usepackage{algorithmic} 

\mathchardef\isinpunto="0010

\newcommand{\abs}[1]{\lvert#1\rvert}

\def\bbbr{{\rm I\!R}}

\theoremstyle{plain}
\newtheorem{proposition}{\bf Proposition}
\newtheorem{theorem}{\bf Theorem}

\newtheorem{corollary}{\bf Corollary}
\theoremstyle{definition}
\newtheorem{definition}{\bf Definition}
\newtheorem{example}{\bf Example}
\newtheorem{remark}{\bf Remark}

\title{\baselineskip .2in
\bf Measuring leadership and productivity in an organisational structure}

\begin{document}

\setlength{\abovedisplayskip}{5mm} 
\setlength{\abovedisplayshortskip}{3mm} 
\setlength{\belowdisplayskip}{6mm} 
\setlength{\belowdisplayshortskip}{5mm} 

\maketitle \vspace*{-1cm} {\baselineskip .2in
\begin{center}
{\small
\begin{tabular}{l}
\bf Ram\'on Flores
\\ Instituto de Matem\'aticas de la Universidad de Sevilla (IMUS), Departamento de Geometr\'{i}a y Topolog\'{i}a,
\\ Universidad de Sevilla,  Spain
\\ e-mail: \url{ramon.flores@uam.es}
\\[2mm]
\bf Elisenda Molina
\\ Instituto de Matem\'atica Interdisciplinar (IMI), Departamento de Estad\'{\i}stica e Investigaci\'on Operativa,
\\ Universidad Complutense de Madrid, Spain
\\ e-mail: \url{elisenda.molina@ucm.es}
\\[2mm]
\bf Juan Tejada\footnote{Proyectos}
\\ Instituto de Matem\'atica Interdisciplinar (IMI), Departamento de Estad\'{\i}stica e Investigaci\'on Operativa,
\\ Universidad Complutense de Madrid, Spain
\\  e-mail: \url{jtejada@mat.ucm.es}
\end{tabular}}
\end{center}
}

\begin{abstract}
\noindent

This paper develops a novel methodological framework for assessing leadership potential and  productivity within organisational structure represented by directed graphs. In this setting, individuals are modeled as nodes and asymmetric supervisory or reporting relationships as directed edges. Leveraging the theory of transferable utility cooperative games, we introduce the Average Forest (AF) measure, a marginalist leadership measure grounded in the enumeration of maximal spanning forests, where teams are hierarchically structured as arborescences. The AF measure captures each agent`s expected contribution across all feasible team configurations under the assumption of superadditivity of the underlying game. We further define a measure of organisational productivity as the expected aggregate value derived from these configurations. The paper investigates key theoretical properties of the AF measure -such as linearity, component feasibility, and monotonicity- and analyzes its sensitivity to structural modifications in the underlying digraph. To address computational challenges in large networks, a Monte Carlo simulation algorithm is proposed for practical estimation. This framework enables the identification of structurally optimal leaders and enhances understanding of how network design impacts collective performance.

\smallskip

\noindent{\bf Keywords:} TU game, directed hierarchical structure, marginal contribution vector, average tree solution

\smallskip

\noindent{\bf JEL Classification Number:} C71

\noindent{\bf Mathematics Subject Classification 2000:} 91A12, 91A43
\end{abstract}

\section{Introduction and Motivation}
\label{intro section}

The measurement of leadership skills in individuals is an important topic in organisational studies, which has generally been conducted through questionnaires and qualitative techniques (Ba\-tis\-ta-Foguet et al., 2021), in an attempt to explore the potential of such individuals. Our approach, however, considers leadership evaluation of an individual by taking into account their current position in the organisational structure, as well as their potential to add value to the formation of working teams that they could lead. To this end, we will consider the organisation to be represented by its organisational chart, i.e. a directed graph, and the possible synergies of forming different working teams as a value function reflecting the potential benefits that the organisation could obtain by assigning a project to that team.  In our approach, it is fundamental to establish which type of working team we should consider. We understand that a working team facing a concrete project must be hierarchically structured.

Such a hierarchical structure has long been associated with increased efficiency in project execution due to its capacity to streamline communication, clarify responsibilities, and optimize decision-making. In hierarchical structures, the flow of information follows a well-defined path, reducing ambiguity and minimising the risk of redundant or conflicting efforts. Such structure ensures that each team member knows their role and to whom they report, which enhances accountability and allows for a more straightforward delegation of tasks (Mintzberg, 1983). Moreover, hierarchical structures are particularly well-suited to projects that require high levels of coordination across different functional areas, as they establish authority gradients that help in managing interdependencies and aligning sub-teams toward common goals (Galbraith, 1973). This clarity becomes even more vital in complex or large-scale projects, where inefficiencies stemming from unclear reporting chains can lead to costly delays or misallocation of resources.

In relation with the problem of team formation in expertise social networks, the advantages of hierarchy directly support the core assumptions of the model presented. The assumption that working teams are organised in hierarchies resonates with findings that structured leadership fosters cohesion and enhances performance, particularly when team tasks are interdependent and time-bound. Furthermore, empirical evidence suggests that hierarchical arrangements promote stability in collaboration patterns, enabling agents to form effective subgroups under recognized leadership (Ahuja, Soda, \& Zaheer, 2012). When leadership is tied both to the structural position within a digraph and the actual productivity of the teams led, as proposed in our model below, the hierarchical organisation not only becomes a governance mechanism but also a driver of measurable performance. By embedding leadership valuation into the network structure, the model aligns with observed organisational dynamics where authority and competence interact to form efficient and adaptive teams.

In this work, we focus on organisational structures in which there is no apparent hierarchical order at first, but which must be hierarchically organised through working teams to develop a project in order to optimize the desired performance.

To be specific, we model an organisation by a digraph, where the nodes represent the members of the organisation, and the arcs reflect the non-symmetric relation among them: $(i,j)$ represents the ability of agents $i$ and $j$ for collaborating effectively under player $i$'s supervision. Observe that any node with several incoming arcs in the digraph may be interpreted as a competition among the superiors/managers for the leadership on the subordinate located at the node. Each player located at a node with several incoming arcs may chose which of his immediate predecessors, who offer the player to join their teams, to follow, but he can follow only one of them. Formally such an assumption is equivalent to the statement that a digraph under scrutiny can be represented by any of its spanning forests, when all the spanning forests are feasible.

Given prior information about the performance of the respective working teams, our aim is to evaluate the efficiency of the structure in terms of its capacity to form working teams. Additionally, we seek to identify the individuals best suited to act as leaders of these working teams. This is achieved by measuring their {\em leadership ability}, which is related to each agent's position in the organisational structure, but also with the performance of each working team that each agent is able to lead. 
 We propose a {\em marginalistic}  leadership valuation which is based on four main assumptions: $(i)$ working teams are organised in hierarchies, $(ii)$ the worth of a working team depends only on its members, $(iii)$ maximal hierarchies are the most productive ones, and $(iv)$ all maximal hierarchies are equally probable to form. Assumptions $(ii)$ and $(iii)$ support to describe each working group's worth by means of an appropriate superadditive TU game, whereas assumption $(iv)$ leads to a {\em non-informative a priori} valuation of each agent's leadership. Theoretically, we define the Average  Forest (AF) measure  and we study its properties. With respect to its properties, special attention is devoted to the right notion of efficiency in this setting, and to the concept of network productivity, as a measure of the ability of the structure to promote the formation of productive teams. We also study the behavior of the Average Forest measure when an existing arc is removed, or a non existing arc is added.

Notwithstanding the fact that the objective of this study is not to define a value for games on digraphs, but rather to measure leadership in the restricted context of organisational situations, the extension of which to more general communication situations may not make sense. The proposed leadership measure does follow a marginalist approach inspired on {\em Average Tree}-like values for graph games, such as the {\em Average tree solution} introduced by Herings {\it et al.} (2008) for cycle-free graph games and generalised in Herings {\it et al.} (2010) to the class of all graph games. Later on, Baron {\it et al.} (2011) elaborates on Average Tree solution for general games and extended  the previous works. This approach had also been adopted in the context of digraph games when the underlying digraph is a tree-digraph in Demange (2004), Khmelnitskaya (2010), and van~den~Brink ,Herings, van~der~Laan and Talman (2015); and for general digraphs in 
 Khmelnitskaya, Sel\c{c}uk and Talman (2020), who introduce the {\em Average Covering Tree value}.  
 
It should be remarked here that if the Average Covering Tree value for general digraph games is used as a measure of leadership, the results obtained are quite different to the ones we get. This is because the model considered in Khmelnitskaya et al. (2020) is very different from ours.  It does not represent appropriately the organisational situation we are interested in. In our model, when we evaluate how productive the organisation is as a whole or how good as leaders are their members, arcs (supervisory relationships) that do not exist in the original organisational structure never appear.   Note that our proposal is not to define a value for digraph-games, that is, for  
  cooperative games in which restrictions in the communications are given by a directed graph.  Instead, we want to measure a specific characteristic of the organisation structure, such as leadership and efficiency.  
  The reader is referred to Gavil\'an, Manuel and van den Brink (2023) for a summary of the different proposals  of directed-graph games and values in the literature: values based on permission structures and hierarchies (Gilles
et al. (1992), Gilles and Owen (1994), van den Brink and Gilles (1996), van den Brink (1997), van den Brink (2017), Algaba and van den Brink (2021), and other values which differ among them on the type of connectivity in the digraph used to define the set of feasible coalitions (Khmelnitskaya, Sel\c{c}uk and Talman (2016), Li and Shan (2020), Gavil\'an et al. (2023).

This article is structured as follows. Section 2 provides all the preliminary concepts from game theory and graph theory necessary to understand the rest of the article. Section 3 begins with the definition of the organisational situation, which constitutes the fundamental context in which our work is developed; it then introduces the Average Forest measure, whose design allows us to evaluate, within an organisational situation, the ability of individuals to lead teams. The main properties of this measure are also described. Section 4 is devoted to the notion of productivity, defined as an average of Average Forest measures, which assesses the ability of the organisational structure to promote the formation of productive working teams. After the definition, several examples and properties of this measure are presented. Section 5 addresses a particularly important phenomenon in this framework: the effect that changes in hierarchical relationships within the organisational situation have on leadership valuations. In Section 6 a Monte Carlo simulation algorithm is proposed to estimate the Average Forest measure in large organisations. The final section concludes the paper.

\section{Preliminaries}
\label{preliminaries section}

In this section we shall introduce some concepts of game theory and graph theory.

\subsection{Transferable utility games}

A \emph{cooperative game with transferable utility} (\emph{TU game}) is a pair $(N,v)$, where $N=\{1,\ldots,n\}$ is a finite set of $n$ players with $n\ge 2$,  and $v\colon 2^N\to\bbbr$ is a \emph{characteristic function} defined on the power set of $N$, satisfying $v(\emptyset)=0$. A subset $S\subseteq N$ is a \emph{coalition} and the associated real number $v(S)$
represents the \emph{worth} of coalition $S$. 

For simplicity of notation and if no ambiguity appears we write $v$ instead of $(N,v)$ when we refer to a TU game. 

The following two properties of TU-games will be crucial throughout the paper:

\begin{itemize}

\item A game
$v\in{\cal G}_N$ is \emph{superadditive} if $v(S\cup T)\ge
v(S)+v(Q)$ for all $S,Q\subseteq N$, such that $S\cap Q=\emptyset$. When the inequality is always an equality the game is called \emph{additive}.

\item  A game $v\in{\cal G}_N$ is \emph{convex} if $v(S\cup Q)+v(S\cap
Q)\ge v(S)+v(Q)$, for all $S,Q\subseteq N$.

\end{itemize}

A \emph{dummy player} is a player $i\in N$ for which $v(S\cup {i})=v(S)+v({i})$ for all $S\subseteq N$.

\subsection{Graphs}

A \emph{graph} on $N$, $\Gamma=(N,E)$, consists of $N$ as a set of nodes and a collection of unordered pairs of nodes $E\subseteq \{\,\{i,j\}\mid i,j\in N,\ i\ne j\}$ as the set of {\em edges}. A \emph{digraph} on $N$, $\Gamma=(N,A)$, consists of $N$ as a set of nodes and a collection of ordered pairs of nodes $A\subseteq\{(i,j)\mid i,j\in N,\ i\ne j\}$ as the set of {\em arcs}.



For a digraph $\Gamma=(N,A)$ on $N$ and a set $S\subseteq N$, the
\emph{subdigraph} of $\Gamma$ on $S$ is the digraph
$\Gamma_S=(S,A_S)$, with $A_S=\{(i,j)\in\!A\mid i,j\!\in S\}$ on $S$. A digraph $\Gamma'=(N,A')$ where $A'\subset A$ is called  \emph{partial digraph} of the digraph $\Gamma=(N,A)$.

 A {\em chain} in a digraph $\Gamma=(N, A)$ is a sequence of nodes $(i_1,i_2,\dots,i_{k-1},i_k)$, $k\geq 2$, and arcs $(a_1,a_2,\dots,a_{k-1})$, such that either $a_{\ell}=(i_{\ell},i_{{\ell}+1})\in A$ or $a_{\ell}=(i_{{\ell}+1},i_{\ell})\in A$, for all $1\leq \ell\leq k-1$. In the sequel, we will refer to a chain as a sequence of nodes without explicit mention of arcs. A {\em path} is an oriented version of a chain in which $(i_{{\ell}},i_{{\ell}+1})\in A$ for any two consecutive nodes $i_{\ell}$ and $i_{{\ell}+1}$ on the path. For $i,j\in N$ if there exists a path from $i$ to $j$ in digraph $\Gamma$ on $N$, then $j$ is a \emph{successor} of $i$ and $i$ is a \emph{predecessor} of $j$ in $\Gamma$. If $(i,j)\in A$, then $j$ is an \emph{immediate successor} of $i$ and $i$ is an \emph{immediate predecessor} of $j$ in $\Gamma$. Let $P^\Gamma(i)$ denote the set of predecessors of $i$ in $\Gamma$, $\widehat P^\Gamma(i)$ the set of immediate predecessors of $i$ in $\Gamma$, $\widehat S^\Gamma(i)$ the set of immediate successors of $i$ in $\Gamma$, $S^\Gamma(i)$ the set of successors of $i$ in $\Gamma$, $\bar P^\Gamma(i)=P^\Gamma(i)\cup\{i\}$, and $\bar S^\Gamma(i)=S^\Gamma(i)\cup\{i\}$. For a node $i\in N$, $\delta^{-}_{\Gamma}(i)=|\widehat P^\Gamma(i)|$ is the \emph{in-degree of $i$} in $\Gamma$. A node $i\in N$ for which $P^\Gamma(i)=\emptyset$ is a
\emph{source} in $\Gamma$.  A {\em cycle} is a chain $(i_1,i_2,\dots,i_{k-1},i_k)$, $k\ge 2$, together with the arc $(i_k,i_1)$ or $(i_1,i_k)$\footnote{Notice that in case $k=2$, if the chain is $(i_1,i_2)$, then the added arc must be the reverse one $(i_2,i_1)$.}. We shall denote a cycle using the notation $(i_1,i_2,\dots,i_{k-1},i_k,i_1)$. Analogously, a {\em circuit} $(i_1,i_2,\dots,i_{k-1},i_k,i_1)$ is a path $(i_1,i_2,\dots,i_{k-1},i_k)$ together with the arc $(i_k,i_1)$.
A digraph is \emph{circuit-free} if it contains no circuits, and it is \emph{cycle-free} if it contains no cycles.



Given a digraph $\Gamma$ on $N$, two nodes $i,j\in N$ are \emph{connected} in  $\Gamma$ if there exists a chain in $\Gamma$ between $i$ and $j$. A digraph $\Gamma$ is \emph{connected}\, if any $i,j\in N$, $i\ne j$, are connected in $\Gamma$. A subset $S\subseteq N$ is \emph{connected}\, in $\Gamma$ if $\Gamma_S$ is connected. For $S\subseteq N$, $C^\Gamma(S)$ denotes the collection of subsets of $S$ connected in $\Gamma$, $S/\Gamma$ is the collection of maximal connected subsets, called \emph{components}, of $S$ in $\Gamma$, and $(S/\Gamma)_i$ is the (unique) component of $S$ in $\Gamma$ containing $i\in S$.

Assume that $\Gamma$ is a weakly connected digraph with no circuits. A \emph{root} $r(\Gamma)$ in $\Gamma$ is a node such that for every other node $i\in N\setminus r(\Gamma)$ there is at least a path in $\Gamma$ from $r(\Gamma)$ to $i$. Moreover, a digraph $\Gamma$ on $N$ is a \emph{arborescence} if it has a unique root such that the path from the root to any other node is unique. A node in an arborescence having no successors is a \emph{leaf}, while an arborescence in which every node has at most one immediate successor is a \emph{path graph}. An arborescence $T$ is a \emph{spanning tree} of a digraph $\Gamma=(N,A)$ if $T$ is a {\em partial digraph of $\Gamma$}. A digraph whose connected components are arborescences is called a \emph{forest}, and is a \emph{spanning forest} if it is a partial digraph.  Notice that a spanning tree is a particular case of a spanning forest with only one component. Moreover, a spanning forest is called \emph{maximal} if adding any other arc of the digraph to it does not give rise to a forest. Let us denote by ${\cal F}^\Gamma$ be the collection of maximal spanning forests of $\Gamma=(N,A)$.

\begin{example}
Figure~\ref{fig-example1} gives an example of a digraph $\Gamma$ on $\{1,2,3,4,5\}$ and its two maximal spanning forests:

\begin{figure}[h]
\begin{center}

\begin{tikzpicture}[scale=1, every node/.style={circle, draw, inner sep=2pt,font=\bf}]

\node [line width=1.2pt] (1) at (10.05,1) {1};
\node [line width=1.2pt] (4) at (9.1,-1) {4};
\node [line width=1.2pt] (3) at (10.6,0) {3};
\node [line width=1.2pt] (2) at (9.6,0)  {2};
\node [line width=1.2pt](5) at (10.1,-1) {5};

\draw[->] [line width=1.2pt] (1) to (2);
\draw[->] [line width=1.2pt] (1) to (3);
\draw[->] [line width=1.2pt] (2) to (4);
\draw[->] [line width=1.2pt] (5) to (2);

\node [line width=1.2pt] (12) at (14.25,1) {1};
\node [line width=1.2pt] (22) at  (13.8,0) {2};
\node [line width=1.2pt] (32) at (14.8,0)  {3};
\node [line width=1.2pt] (42) at (13.3,-1) {4};
\node [line width=1.2pt](52) at (14.8,-1) {5};

\draw[->] [line width=1.2pt] (12) to (32);
\draw[->] [line width=1.2pt] (12) to (22);
\draw[->] [line width=1.2pt] (22) to (42);

\node [line width=1.2pt] (13) at (17.55,1) {1};
\node [line width=1.2pt] (33) at  (17.55,0) {3};
\node [line width=1.2pt] (23) at (18.95,0)  {2};
\node [line width=1.2pt] (43) at (18.95,-1) {4};
\node [line width=1.2pt](53) at (18.95,1) {5};

\draw[->] [line width=1.2pt] (13) to (33);
\draw[->] [line width=1.2pt] (53) to (23);
\draw[->] [line width=1.2pt] (23) to (43);

\end{tikzpicture}
\end{center}

\caption{On the left, the digraph $\Gamma$. On the right, the collection ${\cal F}^\Gamma=\{ F_1,F_2\}$ of its maximal spanning forests.}
\label{fig-example1}
\end{figure}
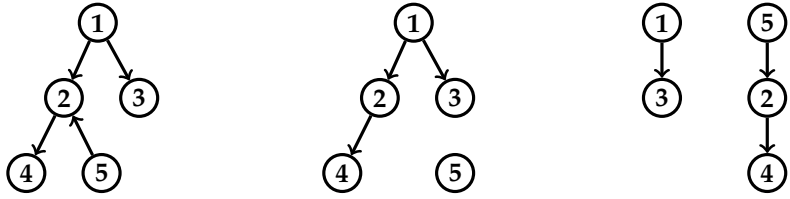

\label{example1}
\end{example}




\section{Measuring leadership in an organisational situation}
\label{AF section}

In this work, we focus on an organisational structure in which there is no apparent hierarchical order at first, but which must be hierarchically organised through working teams in order to optimize the desired performance.
Given prior information about the performance of the respective working teams, our aim is to evaluate the efficiency of the structure in terms of its capacity to form working teams. Additionally, we seek to identify the individuals best suited to act as leaders of these working teams.

In this section, we will introduce a leadership measure for each individual in an organisational situation. This measure will take into account their position in the organisational structure, as well as the value they add when leading a working team.


First of all, given a digraph $\Gamma=(N, A)$, we interpret the arc $(i,j)\in A$ as member $i$ has the potential power of organising the tasks of member $j$, or alternatively, the potential obligation of $j$ to report to $i$. A \emph{organisational structure} is a digraph without circuits.

Given the above interpretation what we shall measure is the ability of each agent to lead working teams. Formally, given a organisational structure $\Gamma=(N,A)$, a \emph{working team} is any partial subdigraph of $\Gamma$ which is an arborescence. A working team is lead by its root.


The value or obtained benefit of a working team is given by the game $(N, v)$, i.e. if $S\subset N$ is the set of members of a working team, its value is given by $v(S)$. That value could be calculated by evaluating the benefits that the group forming the working team can obtain  developing a project. Note that the value only depends on the members of the working team and not of the way they are organised. Moreover, we shall consider superadditive games, assuming that toxic leadership has been previously avoided.

Finally, $(N, v, \Gamma)$, with $\Gamma$ an organisational structure and $v$ superadditive,  will be called an \emph{organisational situation}. The class of organisational situations with set of agents $N$ is denoted by ${\cal G}^N$. In this setting, we define a {\em leadership measure} as a real non-negative function $\xi: {\cal G}^N \rightarrow \mathbb{R}^n$.

\subsection{Average forest leadership measure: definition and properties}

As we said, in this section we introduce a leadership measure on the class of organisational situations for evaluating the ability of each agent to lead working teams. We propose a  {\em marginalistic} measure based on four main assumptions: $(i)$ working teams are organised hierarchically (arborescences), $(ii)$ the worth of a working team depends only on its members, $(iii)$ the set of members $N$ are organised in working teams that form a maximal spanning forest of $\Gamma$ that are the most productive ones, and $(iv)$ all maximal spanning forests are equally probable to form. Assumptions $(ii)$ and $(iii)$ support to describe each working group's worth by means of an appropriate superadditive TU game. Assumptions $(i)$ and $(iii)$ lead the maximal spanning forest to be the optimal configurations of the social network in working teams, whereas
 assumption $(iv)$ leads to a {\em non-informative a priori} valuation of each agent's leadership.


Based on these ideas, we define the \emph{marginal contribution vector} $m^{(v,\Gamma)} (F)\in\bbbr^N$ with respect to a forest $F\in {\cal F}^{\Gamma}$  as the vector:
\begin{equation*}\label{mc-tv}
m_i^{(v,\Gamma)}(F)=v(\bar S^F(i))-
       \sum\limits_{j\in\widehat S^F(i)}v(\bar S^F(j)),
               \qquad\mbox{for all}\enspace i\in N,
\end{equation*}
or equivalently,
\begin{equation*}\label{mcv}
m_i^{(v,\Gamma)}(F)=v(\bar S^F(i))-
       \sum\limits_{C\in S^F(i)/F}v(C),
               \qquad\mbox{for all}\enspace i\in N.
\end{equation*}

In the marginal contributions' vector that corresponds to a forest $F$, a player receives as payoff the difference between the worth of the maximal working team that he can lead and the total worths of the subteams in which this maximal team  splits when he quits. This difference is the contribution of the player when he joins his successors (in the arborescence) to lead the created working team. Now, if we do not have any prior information about the probability of formation of each of the maximal spanning forests $F\in {\cal F}^{\Gamma}$, we could assume that all of them are equally probable, and propose as a priori evaluation of each player leadership the following measure.

\begin{definition}
For an organisational situation $(N, v, \Gamma)\in {\cal G}^N$, the
\emph{average forest leadership measure} (\emph{AF measure}) is the average of the marginal contribution vectors corresponding to all maximal spanning forests of the digraph $\Gamma$, i.e.,
$$
AF(N, v,\Gamma)=\frac{1}{|{\cal F}^\Gamma|}
     \sum\limits_{F\in{\cal F}^{\Gamma}}m^{(v,\Gamma)}(F).
$$
\label{AFV-definition}
\end{definition}

Due to the superadditivity of $v$, the $AF$ measure is non-negative and is therefore a leadership measure.

\begin{example}
\label{example2}
In Example \ref{example1}, the AF measure is 
$AF_1(N,v,\Gamma)=\frac{1}{2}(v(\{ 1,2,3,4\})-v(\{2,4\})-v(3)) + \frac{1}{2}(v(\{ 1,3\})-v(3))$,
$AF_2(N,v,\Gamma)=v(\{ 2,4\})-v(4)$, $AF_3(N,v,\Gamma)=v(3)$, $AF_4(N,v,\Gamma)=v(4)$, and
$AF_5(N,v,\Gamma)=\frac{1}{2}v(5) + \frac{1}{2}(v(\{ 5,2,4\})-v(\{2,4\}))$. For instance, for the attachment game $v(S)=\abs{S}-1$, the $AF$ vector is $(3/2, 1, 0, 0, 1/2)$.
\end{example}

Note that if the organisational structure $\Gamma$ is a forest, then the Average Forest measure coincides with the {\em Tree value for forest digraph games}, introduced in Demange (2004) and axiomatised in Khmelnitskaya (2010).

In the sequel we will analyse several properties of the $AF$ measure.

 A leadership measure $\xi$ on ${\cal G}^N$ is \emph{linear} if for any two organisational situations $(N, v,\Gamma)$, $(N, w,\Gamma)\in{\cal G}^N$ and all $a,b\in\bbbr$, it holds that
\begin{equation*}
\xi(N,av+bw,\Gamma)=a\xi(N,v,\Gamma)+b\xi(N, w,\Gamma),
\end{equation*}
where $av+bw$ is defined as $(av+bw)(S)=av(S)+bw(S)$ for all $S\subseteq N$.


In order to define a \emph{dummy-player property}, we must first define properly what a {\em dummy player} is in an organisational situation. Obviously, the information given by the game is not enough to get a good definition, as a player can be dummy in the game, but not in the organisational situation by leading value added working teams.
Conversely, a player can be a non-dummy player in $(N, v)$, but in the organisational situation to have no opportunity of merging working teams to which he could add value. 



Therefore, an \emph{organisational situation dummy player} is a player $i$ for which $m_i^{(v,\Gamma)}(F)=v({i})$, for all $F\in {\cal F}^{\Gamma}$. 

The above definition means that a dummy player only can merge non synergistic working teams in any forest to which he/she does not add any value. Obviously the leafs are dummies.



In these conditions, we say that a leadership measure $\xi$ on ${\cal G}^N$ satisfies the \emph{organisational situation dummy-player property} if for any organisational situation $(N, v,\Gamma)\in{\cal G}^N$ it holds that $\xi_i(N, v,\Gamma)=v({i})$ for any organisational situation dummy player $i$. 

Note that a player $i\in N$ such that he/she  and all its subordinates in $\Gamma$ are dummy players in $(N,v)$, is also an organisational situation dummy player. Thus, dummy-player property implies inessential player property of van den Brink et al. (2015). In fact, it is a stronger property, since not all inessential player in the sense of van de Brink et al. (2015) is an organisational situation inessential player.


It is easy to see that the linearity and the organisational situation dummy-player property of the average forest measure follow straightforwardly from its definition.

Now, a leadership measure $\xi$ is \emph{component efficient} (CE) if for any organisational situation $(N,v,\Gamma)\in {{\cal G}^N}$, and for all $C\in N/\Gamma$ it holds that $$x(C):=\sum_{i\in C}\xi_i(v,\Gamma)=v(C),$$ and \emph{component feasible} if $x(C)\le v(C)$ for all $C\in N/\Gamma$.
  It is clear that the AF measure fails to verify component efficiency. However, as the underlying game is superadditive, the worth $v(C)$ of a connected component $C$ is an upper bound for its total value, i.e., it is component feasible. In this setting, when coalition $S$ forms, their agents look for the best possible arrangement of themselves on working teams, which must be arborescences;  therefore, those working teams constitute a partition which is a refinement of the partition determined by its connected components. Since we are interested on the case in which the grand coalition forms, we obtain the following result.

\begin{proposition}\label{efficiency-rationality}
The average forest measure is component feasible, i.e.,
$$
\sum_{i\in C}AF_i(N,v,\Gamma)\leq v(C),
$$
for every connected component $C$ of $\Gamma$.
\end{proposition} 


\begin{proof}
For any $(N, v,\Gamma)\in{\cal G}^N$ the average forest measure $AF(N,v,\Gamma)$ is the average of the marginal contribution vectors that correspond to all spanning forests of the digraph $\Gamma$. Let $F\in{\cal F}^\Gamma$. By definition of forest, $F$ provides a partition of $N$ given by a set of coalitions ${\cal P} (F)=\{N_{k}\}_{k=1}^m$, such that $F|_{N_{k}}$ is an arborescence on $N_{k}$, $k=1,...,m$. From the definition of the marginal contribution vector corresponding to a forest we easily obtain that
 $$
 \sum_{i\in N_{k}} m_i^{(N,v,\Gamma)}(F)=v(N_{\ell}),
               \quad\mbox{for all}\enspace k=1,\dots,m.
 $$
Whence, taking into account that partition ${\cal P} (F)$ must be a refinement of $N/\Gamma$, and together with superadditivity of $v$, we get that for every component $C\in N/\Gamma$ of $N$ in $\Gamma$, $C=\cup_{k\in {\cal K}} N_k$, where ${\cal K}\subseteq \{1,\dots,m\}$, and therefore
 $$
 \sum_{i\in C}m_i^{(v,\Gamma)}(F)= \sum_{k\in {\cal K}} \sum_{i\in N_k} m_i^{(N,v,\Gamma)}(F)=
        \sum_{k\in {\cal K}}v(N_{\ell})\le v(C),
 $$
i.e., every marginal contribution vector corresponding to a forest provides a component feasible payoff.
\end{proof}

From this result, it follows that the measure does not have to be efficient when the structure is not efficient. See, for instance, Example \ref{example2}. We shall consider this issue in section \ref{secprod}.  

Finally, we consider monotonicity properties of the $AF$ measure. Monotonicity is a key property in measurement, as it reflects the expectation that the value of a measure should increase when the relevant conditions or inputs improve. Here we shall consider some monotonicity conditions that appear in the game theory literature (see, for instance (Young, 1985).

\emph{Coalitional Monotonicity (CM)}: given an organisational situation $(N, v, \Gamma)$, a leadership measure $\xi$ is coalitionally monotonic if for a fixed $\Gamma$, $\xi_i(N,v,\Gamma)\leq \xi_i(N,w,\Gamma)$, for $i\in T \subseteq N$, when (i) $v(S)=w(S)$ for all $S\neq T$ and (ii) $v(T) < w(T)$.

\emph{Strong Monotonicity (SM)}: given an organisational situation $(N, v, \Gamma)$, a leadership measure $\xi$ is strongly monotonic if for $v(S\cup \{i\})-v(S) \leq w(S\cup \{i\})-w(S)$, for all $S\subseteq N\setminus \{i\}$, then $\xi_i(N,v,\Gamma)\leq \xi_i(N,w,\Gamma)$.

\emph{Individual Monotonicity (IM)} Given an organisational situation $(N, v, \Gamma)$, a leadership measure $\xi$ is individually monotonic if for $v(S)=w(S)$ for all $S\subseteq N\setminus \{i\}$ and $v(S\cup \{i\})\leq w(S\cup \{i\})$, for all $S\subseteq N\setminus \{i\})$, then $\xi_i(v)\leq \xi_i(w)$.

The $AF$ measure trivially satisfies CM, but not SM: assume the condition of SM is fulfilled and assume also that there exist two working teams $T_1$ and $T_2$ in $\Gamma$ that can be coordinated by $i$. Then the inequality $v(T_1\cup T_2\cup \{i\})-v(T_1)-v(T_2)\leq w(T_1\cup T_2\cup \{i\})-w(T_1)-w(T_2)$ cannot be guaranteed as it also depends on the values of the measures of the working teams in the games $v$ and $w$. Conversely, that inequality can be guaranteed by IM.

\section{Measuring productivity of an organisational situation}
\label{secprod}

Taking into account the possible inefficiency of the average forest measure, we can consider that can be organisational structures that are more or less productive. From this point of view, Proposition \ref{efficiency-rationality} gives an upper bound for the total benefit that the entire group $N$ can achieve. Now, if we want to give {\em a priori evaluation} of the ability of the organisational structure to promote the formation of productive working teams, we can assume that every possible partition of the entire group $N$ in maximal hierarchies is equally probable and define the {\em productivity of the digraph} $\Gamma$, given the TU game $(N,v)$ as follows, where the benefit that the entire group $N$ can achieve when they are organised as the spanning forest $F\in {\cal F}^{\Gamma}$ is given by the sum of the value that each of its hierarchies can get.

\begin{definition}
 Let $(N,v,\Gamma)\in{\cal G}^N$; then, the  {\em productivity of the organisational situation} $(N,v,\Gamma)$ is defined to be
$$
Prod(N,v,\Gamma) = \frac{1}{|{\cal F}^\Gamma|}\sum_{F\in {\cal F}^{\Gamma}} \sum_{K\in N/F} v(K) .
$$
\end{definition}

Note that the value that the Average Forest measure assigns to the entire group $N$,
\begin{equation}
\displaystyle
\sum_{i\in N} AF_i(N,v,\Gamma)=Prod(N,v,\Gamma)
 \label{sumAVF-productivity}
 \end{equation}
 In fact, this is precisely the reason for failing component efficiency. In this framework, component efficiency is in general a non affordable goal, that depends on the productivity of the relational structure. Let us analyse some examples.

\begin{example}
Let us consider the two following structures, $\Gamma^1, \Gamma^2$, depicted in Figure \ref{productivity-example1}:
\begin{figure}[h]
\begin{center}
\begin{tikzpicture}[scale=1, every node/.style={circle, draw, inner sep=2pt,font=\bf}]

\node[line width=1.2pt] (1) at (0.75,2.4) {1} ;
\node[line width=1.2pt] (4) at (1.5,0) {4} ;
\node[line width=1.2pt] (2) at (1.5,1.2) {2} ;
\node[line width=1.2pt] (3) at (2.25,2.4) {3} ;

\draw[->] [line width=1.2pt] (1) to (2);
\draw[->] [line width=1.2pt,color=red] (3) to (2);
\draw[->] [line width=1.2pt] (2) to (4);

\node[line width=1.2pt] (1) at (4.75,2.4) {1} ;
\node[line width=1.2pt] (4) at (5.5,0) {4} ;
\node[line width=1.2pt] (2) at (5.5,1.2) {2} ;
\node[line width=1.2pt] (3) at (6.25,2.4) {3} ;

\draw[->] [line width=1.2pt] (1) to (2);
\draw[->] [line width=1.2pt,color=red] (2) to (3);
\draw[->] [line width=1.2pt] (2) to (4);

\end{tikzpicture}
\end{center}
\caption{On the left, organisational structure $\Gamma^1$. On the right, organisational structure $\Gamma^2$}
\label{productivity-example1}
\end{figure}
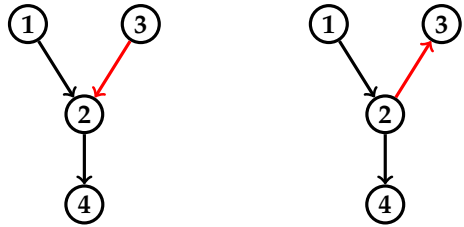

In this example the two organisational structures are very close. However,  $\Gamma^1$ is less productive that it is $\Gamma^2$, since competition among agents $1$ and $3$ over the leadership of $2$ lead one of them to remain out of the working group. The organisational structure $\Gamma^2$ achieves the maximum productivity of $v(N)$, whereas both maximal forest of  $\Gamma^1$ are less productive:
$$
Prod(N,v,\Gamma^1) = \frac{1}{2} \bigl ( v(1)+v(\{2,3,4\}\bigr ) + \frac{1}{2} \bigl ( v(3)+v(\{1,2,4\}\bigr ) \leq v(N).
$$
In this example is not possible to organise the members of the unique connected component of $\Gamma^1$ in an effective working hierarchy without splitting them in two independent working groups.  
\end{example}

The following theorem provides a necessary and sufficient condition for the efficiency of the Average Forest measure, i.e. for the organisation to reach its maximum productivity. First, we recall the definition of quasi-strongly connected digraph.

\begin{definition}[Zeng et al. (2015)] Let $\Gamma$ be a weakly connected digraph with no circuits. Then $\Gamma$ is said to be {\em quasi-strongly connected} if it has a root, or equivalently, if it has a maximal spanning tree.   
\end{definition}

Observe that the root must be unique. In turn, the following equivalence is very easy to check:

\begin{proposition}
\label{rootarbo}
In the previous conditions, $\Gamma$ is quasi-strongly connected if every maximal forest of $\Gamma$ is an arborescence.   
\end{proposition}

Now we can state the following result:

\begin{theorem}\label{th:2}
Let $(N,v,\Gamma)\in{\cal G}^N$. If $\Gamma$ is a quasi-strongly connected digraph, then $Prod (N,v,{\Gamma})=v(N)$. Conversely, it is also a necessary condition if $v$ is non-additive, up to component additivity .
\end{theorem}

\begin{proof} 
Every forest is a partial digraph. In particular, as $\Gamma$ is quasi-strongly connected, every maximal forest is an arborescence by Proposition \ref{rootarbo}. Then, the AF measure of each forest is $v(N)$, and the productivity $Prod (N,v,{\Gamma})$, which is an average of all these numbers, should be again $v(N)$.

Conversely, assume that $v$ is non-additive, up to component additivity, and argue by contradiction that $\Gamma$ is not quasi-strongly connected. By Proposition \ref{rootarbo}, there exists a maximal forest in $\Gamma$ which possesses at least two arborescences. This implies, by the definition of productivity, that the average of the values of the arborescences is strictly smaller than $v(N)$. This concludes the proof.

\end{proof}

\begin{remark}

Observe that if the game is additive, the productivity is always $v(N)$, and in this case the necessity does not hold when the graph is not quasi-strongly connected.

\end{remark}

The importance of Theorem \ref{th:2} is that an efficient organisational structure does not require a hierarchical one, but rather a unique head from which a unique reporting chain extends to every other member of the organisation. 

\begin{example}
   The following organisational structure attains maximum productivity $v(N)$ without being an arborescence:

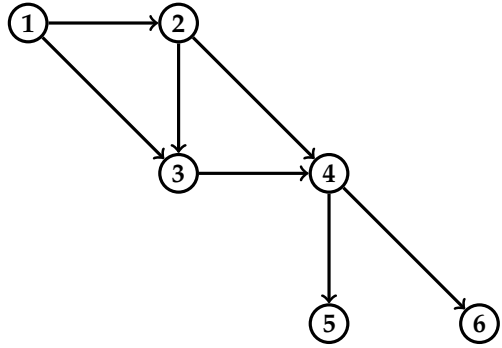
\begin{figure}[h]
\begin{center}
\begin{tikzpicture}[node_style/.style={draw,circle,minimum size=0.5cm,inner sep=1,font=\bf}]

\node[node_style] [line width=1.2pt] (uno) at (0,2) {1} ;
\node[node_style] [line width=1.2pt] (dos) at (2,2) {2} ;
\node[node_style] [line width=1.2pt] (tres) at (2,0) {3} ;
\node[node_style] [line width=1.2pt] (cuatro) at (4,0) {4} ;
\node[node_style] [line width=1.2pt](cinco) at (4,-2) {5} ;
\node[node_style] [line width=1.2pt](seis) at (6,-2) {6} ;

\draw[->] [line width=1.2pt] (uno) to (dos);
\draw[->] [line width=1.2pt] (uno) to (tres);
\draw[->] [line width=1.2pt] (dos) to (tres);
\draw[->] [line width=1.2pt] (dos) to (cuatro);
\draw[->] [line width=1.2pt] (tres) to (cuatro);
\draw[->] [line width=1.2pt] (cuatro) to (cinco);
\draw[->] [line width=1.2pt] (cuatro) to (seis);
\end{tikzpicture}
\end{center}
\caption{Maximum productivity}
\label{productivity-example-max}
\end{figure}

\end{example}

\section{The effect of eliminating reporting relationships}
\label{delarc}

In this section we study the impact over the Average Forest measure when an arc is deleted. The results can also be used to study the effect of adding an arc $(i, j)$, which means deciding that $i$ must report to $j$. First we will check that it fails to verify the generalisation to this context of the property of inessential arc.



Given a digraph $\Gamma$ on $N$, an arc $(i,j)\in
\Gamma$ is an \emph{inessential arc} if $i\notin
S_{\Gamma}(j)$ and there exists $i'\in N$ such that $(i,i')\in\Gamma$, $i\notin S_{\Gamma}(i'),$ and $j\in S_{\Gamma}(i')$.
In words, an arc $(i,j)\in \Gamma$ is inessential if it is possible to
reach node $j$ from $i$ also by using a chain different  than the arc $(i,j)$. The absence of an inessential arc does not change the set of predecessors of any player.

A leadership measure $\xi$ on ${\cal G}_N^\Gamma$ possesses the \emph{inessential arc property} if for any digraph game $(v,\Gamma)\in{\cal G}_N^\Gamma$ and inessential arc $(i,j)\in\Gamma$ it holds that $\xi(N,v,\Gamma)= \xi(v,\Gamma \setminus \{(i,j)\})$.

The AF measure fails to verify the inessential arc property, and this fact is meaningful in this setting. Note that previous definition of inessential arc is a mathematical generalization of an inessential edge in a graph, in which the goal is to be connected regardless of the way in which connection is hold. This is not the case in the setting of team formation, where the specific reporting chain is crucial for the future of the team. In this framework, Notably over the Average Forest measure of each players deserves a carefully study. All proofs rely on how the collection ${\cal F}^{\Gamma}$ is constructed. Note that in this case the set  of sources in $\Gamma$ is given by $So (\Gamma) =\{ i\in N\, \vert \, \delta^{-}_{\Gamma} (i)=0\}$. Then,
\begin{equation}  
\vert {\cal F}^{\Gamma} \vert = \displaystyle \prod_{i \notin So (\Gamma)} \delta^{-}_{\Gamma} (i),
\label{cardinal-max-forests}
\end{equation}
and each of the maximal spanning forests $F\in {\cal F}^{\Gamma}$ is obtained by means of selecting one leader out of the set $\hat{P}_{\Gamma} (i)$ per each $i\notin So (\Gamma)$. This can be formally proved by induction. 

Note that this argument provides an algorithm to explicitly construct the maximal spanning forests of a graph, as illustrated in the following example.

\begin{example}
Let us consider for instance the following digraph $\Gamma$:
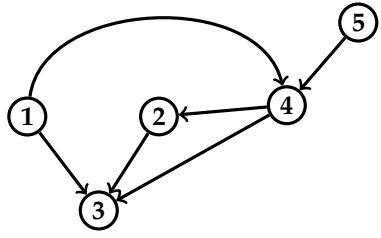
\begin{center}
\begin{tikzpicture}[scale=1, every node/.style={circle, draw, inner sep=2pt,font=\bf}]

\node[line width=1.2pt] (1) at (-0.2,1.75) {1} ;
\node[line width=1.2pt] (5) at (4.2,3.0) {5} ;
\node[line width=1.2pt] (4) at (3.25,1.90) {4} ;
\node[line width=1.2pt] (2) at (1.55,1.75) {2} ;
\node[line width=1.2pt] (3) at (0.75,0.5) {3} ;

\draw[->] [line width=1.2pt] (1) to (3);
\draw[->] [line width=1.2pt] (2) to (3);
\draw[->] [line width=1.2pt] (4) to (3);
\draw[->] [line width=1.2pt] (4) to (2);
\draw[->] [line width=1.2pt] (5) to (4);

\draw [line width=1.2pt] (1) edge[->,bend left=80]  (4);

\end{tikzpicture}

\end{center}
Then, the family ${\cal F}^{\Gamma}$ of spanning forest of the digraph $\Gamma$, depicted in Figure \ref{short-example-forests-family-figure}, is made up of six spanning forest $F_{i,j}$, where $i\in \hat{P}_{\Gamma} (4)$ denotes the player who leaders the work of player 4 in the corresponding hierarchy, and being $j\in \hat{P}_{\Gamma} (3)$ the player who leaders the work of 3.

\vspace*{3mm}

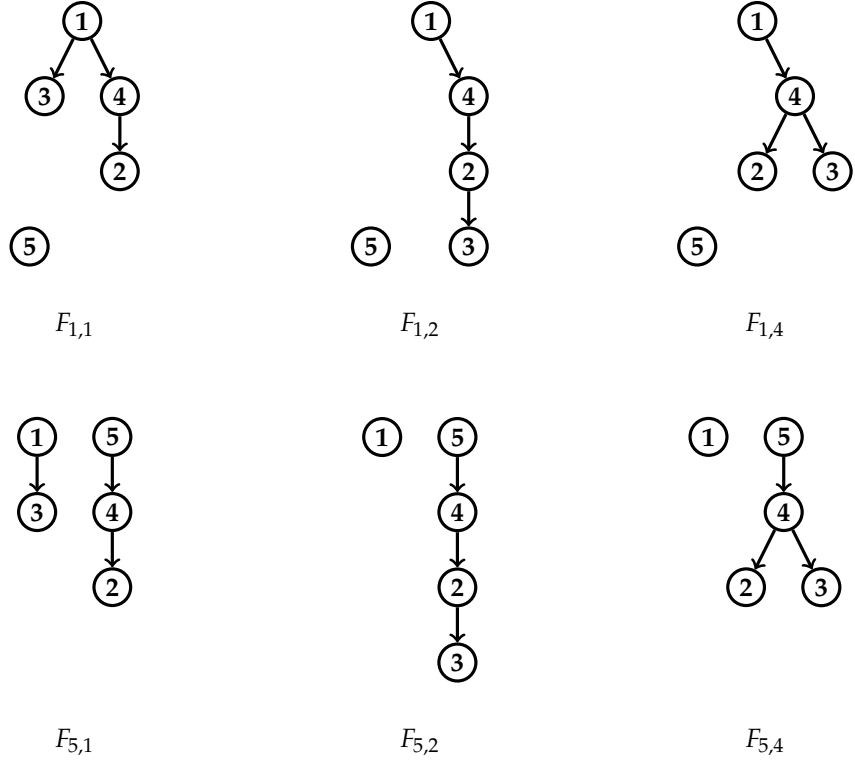
\begin{figure}[htbp]
\centering

\begin{minipage}{0.3\textwidth}
\centering
\begin{tikzpicture}[scale=1, every node/.style={circle, draw, inner sep=2pt,font=\bf}]
\node [line width=1.2pt] (1) at (1,3) {1};
\node [line width=1.2pt] (4) at (1.5,2) {4};
\node [line width=1.2pt] (3) at (0.5,2) {3};
\node [line width=1.2pt] (2) at (1.5,1) {2};
\node [line width=1.2pt](5) at (0.3,0) {5};

\draw[->] [line width=1.2pt] (1) to (3);
\draw[->] [line width=1.2pt] (1) to (4);
\draw[->] [line width=1.2pt] (4) to (2);
\end{tikzpicture}
\[
F_{1,1}
\]
\end{minipage}
\begin{minipage}{0.3\textwidth}
\centering
\begin{tikzpicture}[scale=1, every node/.style={circle, draw, inner sep=2pt,font=\bf}]
\node [line width=1.2pt] (1) at (0,3) {1};
\node [line width=1.2pt] (4) at (0.5,2) {4};
\node [line width=1.2pt] (2) at (0.5,1) {2};
\node [line width=1.2pt] (3) at (0.5,0) {3};
\node [line width=1.2pt] (5) at (-0.8,0) {5};

\draw[->] [line width=1.2pt] (1) to (4);
\draw[->] [line width=1.2pt] (4) to (2);
\draw[->] [line width=1.2pt] (2) to (3);
\end{tikzpicture}
\[
F_{1,2}
\]
\end{minipage}
\begin{minipage}{0.3\textwidth}
\centering
\begin{tikzpicture}[scale=1, every node/.style={circle, draw, inner sep=2pt,font=\bf}]
\node [line width=1.2pt] (1) at (0,3) {1};
\node [line width=1.2pt] (4) at (0.5,2) {4};
\node [line width=1.2pt] (2) at (0.0,1) {2};
\node [line width=1.2pt] (3) at (1.0,1) {3};
\node [line width=1.2pt] (5) at (-0.8,0) {5};

\draw[->] [line width=1.2pt] (1) to (4);
\draw[->] [line width=1.2pt] (4) to (2);
\draw[->] [line width=1.2pt] (4) to (3);
\end{tikzpicture}
\[
F_{1,4}
\]
\end{minipage}

\vspace{1cm}

\begin{minipage}{0.3\textwidth}
\centering
\begin{tikzpicture}[scale=1, every node/.style={circle, draw, inner sep=2pt,font=\bf}]
\node [line width=1.2pt] (5) at (0.5,4) {5};
\node [line width=1.2pt] (4) at (0.5,3) {4};
\node [line width=1.2pt] (2) at (0.5,2) {2};
\node [line width=1.2pt] (1) at (-0.5,4) {1};
\node [line width=1.2pt] (3) at (-0.5,3) {3};
\node [line width=1.2pt,color=white] (6) at (-0.5,1) {{\color{white} 6}};

\draw[->] [line width=1.2pt] (1) to (3);
\draw[->] [line width=1.2pt] (5) to (4);
\draw[->] [line width=1.2pt] (4) to (2);
\end{tikzpicture}
\[
F_{5,1}
\]
\end{minipage}
\begin{minipage}{0.3\textwidth}
\centering
\begin{tikzpicture}[scale=1, every node/.style={circle, draw, inner sep=2pt,font=\bf}]
\node [line width=1.2pt] (5) at (0.5,4) {5};
\node [line width=1.2pt] (4) at (0.5,3) {4};
\node [line width=1.2pt] (2) at (0.5,2) {2};
\node [line width=1.2pt] (1) at (-0.5,4) {1};
\node [line width=1.2pt] (3) at (0.5,1) {3};

\draw[->] [line width=1.2pt] (2) to (3);
\draw[->] [line width=1.2pt] (5) to (4);
\draw[->] [line width=1.2pt] (4) to (2);
\end{tikzpicture}
\[
F_{5,2}
\]
\end{minipage}
\begin{minipage}{0.3\textwidth}
\centering
\begin{tikzpicture}[scale=1, every node/.style={circle, draw, inner sep=2pt,font=\bf}]
\node [line width=1.2pt] (5) at (0.5,4) {5};
\node [line width=1.2pt] (4) at (0.5,3) {4};
\node [line width=1.2pt] (2) at (0.0,2) {2};
\node [line width=1.2pt] (1) at (-0.5,4) {1};
\node [line width=1.2pt] (3) at (1.0,2) {3};
\node [line width=1.2pt,color=white] (6) at (-0.5,1) {{\color{white} 6}};

\draw[->] [line width=1.2pt] (4) to (3);
\draw[->] [line width=1.2pt] (5) to (4);
\draw[->] [line width=1.2pt] (4) to (2);
\end{tikzpicture}
\[
F_{5,4}
\]
\end{minipage}

\caption{Family of spanning forests. Constructing example.}
\label{short-example-forests-family-figure}
\end{figure}
\label{short-example-forests-family}
\end{example}

We first prove that the AF measure verifies the {\em successor equivalent} (SE) property introduced in (Khmelnitskaya, 2010). Let us recall its definition:


\begin{definition}
A leadership measure $\xi$ is {\em successor equivalent} (SE) if, for any organisational situation $(N,v,\Gamma)\in {\cal G}^N$, and for every arc $(i_0,j_0 )\in A_{\Gamma}$, it holds
\begin{equation}
\xi_k(N,v,\Gamma)=\xi_k(N,v,\Gamma^{-i_0j_0}), \text{ for all $k\in \overline{S}_{\Gamma} (j_0)$,}
\label{SE-property}
\end{equation}
where $\Gamma^{-i_0j_0}=\Gamma\setminus\{ (i_0,j_0)\}$.
\end{definition}

Our goal now is to check the Average Forest measure verifies this property. In order to do so, we will prove a more general result, from which this property will be obtained as a direct consequence.

To state this result, for a given digraph $\Gamma=(N, A_{\Gamma})$ and a vertex $k\in N$ we will define the {\em local digraph associated to $k$}, denoted as $\Gamma(k)=(SP_{\Gamma} (k), A(k))$, in the following way:
\begin{itemize}
\item The vertex set is given by $$SP_{\Gamma} (k) =\overline{S}_{\Gamma} (k) \cup \hat{P}_{\Gamma} (S_{\Gamma} (k))$$
\item The arc set is given by
$$A_{\Gamma}(k) =A_{\Gamma}\vert _{\overline{S}_{\Gamma} (k)} \cup \{ (i,j) \, \vert \, i\in \hat{P}_{\Gamma} (S_{\Gamma} (k)) \text{ and } j\in {S}_{\Gamma} (k)\}.$$
\end{itemize}
 
The following result shows that the Average Forest measure does not change for any vertex such that its local digraph is not affected by deleting an arc $(i_0, j_0)$.

\begin{proposition}
\label{successor}
Let $(N,\Gamma,v)\in {\cal G}^N$, and let $(i_0,j_0)$ be an arc in $A_{\Gamma}$. Then, for every player $k\in N$ for which
$\Gamma(k)=\Gamma^{-i_0j_0} (k)$ the following holds:
\begin{equation}
AF_k (N,v,\Gamma)=AF_k (N,v,\Gamma^{-i_0j_0}).
\label{SPE-equation}
\end{equation}
\label{SPE-proposition}
\end{proposition}

\begin{proof}
For each spanning forest $F\in {\cal F}^{\Gamma}$ let $T_{k} (F)$ be the maximal team of $F$ leaded by $k$, i.e., the sub-arborescence of $F$ rooted in $k$. That is $T_k(F)=F\vert \overline{S}_{F} (k)$. Now, denote by $T$ the team $T_k(F)$ for a fixed $F\in {\cal F}^{\Gamma}$. Then, the proof relies on counting the number of spanning forests $F'$ for which $T_k(F')$ coincides with the given team $T$; this is equal to the product:
$$
 \prod_{\substack{i \notin So(\Gamma) \\ i \notin T}} \delta^{-}_{\Gamma} (i).
$$
Denote by ${\cal T}_k ({\Gamma})$ the collection of sub-arborescences $\{ T_k(F)\, \vert\, F\in {\cal F}^{\Gamma}\}$, and let $(i_0,j_0)\in A_{\Gamma}$ such that  $\Gamma(k)=\Gamma^{-i_0j_0} (k)$. Then, ${\cal T}_k ({\Gamma})={\cal T}_k ({\Gamma}^{-i_0j_0})$,
 $j_0 \notin So ({\Gamma})$ and $j_0 \notin T_k(F)$ for each $T_k(F)\in {\cal T}_k ({\Gamma})$.  Thus, the number of spanning forests $F'\in {\cal F}^{\Gamma^{-i_0j_0}}$ for which $T_k(F')$ coincides with a given $T\in {\cal T}_k ({\Gamma}^{-i_0j_0})$,  is given by the product:
$$
 \prod_{\substack{i \notin So({\Gamma^{-i_0j_0}}) \\ i \notin T}} \delta^{-}_{\Gamma^{-i_0j_0}} (i) =
 \begin{cases}
(\delta^{-}_{\Gamma} (j_0) -1)  \prod_{\substack{i \notin So ({\Gamma}) \\ i \notin T, i\neq j_0}} \delta^{-}_{\Gamma} (i),
& \text{ if $\delta^{-}_{\Gamma} (j_0) > 1$,}
\\[4mm]
\prod_{\substack{i \notin So(\Gamma) \\ i \notin T, i\neq j_0}} \delta^{-}_{\Gamma} (i),
& \text{ if $\delta^{-}_{\Gamma} (j_0) = 1$.}
\end{cases}
$$
Moreover,  $m_k^{(v,\Gamma)} (F)=m_k^{(v,\Gamma^{-i_0j_0})} (F)$ for all $F$ such that $T_k(F)=T$. Therefore, if $\delta^{-}_{\Gamma} (j_0) =1$, then ${\cal F}^{\Gamma^{-i_0j_0}}=\{ F\setminus \{(i_0,j_0)\} \, \vert\, F\in {\cal F}^{\Gamma}\}$ and $AF_k(N,v,\Gamma)=AF_k(N,v,\Gamma^{-i_0j_0})$. Otherwise, if $\delta^{-}_{\Gamma} (j_0) > 1$, it follows
\begin{multline}
AF_k(N,v,\Gamma)=\sum_{T \in {\cal T}_k ({\Gamma})} \frac{ \prod_{\substack{i \notin So({\Gamma}) \\ i \notin T}} \delta^{-}_{\Gamma} (i)}{\prod_{i \notin So({\Gamma})} \delta^{-}_{\Gamma} (i)}
 m^{(v,\Gamma)} (T) =
\\
\sum_{T \in {\cal T}_k ({\Gamma}^{-i_0j_0})} \frac{ (\delta^{-}_{\Gamma} (j_0) -1)  \prod_{\substack{i \notin So({\Gamma}) \\ i \notin T, i\neq j_0}} \delta^{-}_{\Gamma} (i)}
{(\delta^{-}_{\Gamma} (j_0) -1) \prod_{\substack{i \notin So({\Gamma}) \\ i\neq j_0}} \delta^{-}_{\Gamma} (i)}
m^{(v,\Gamma^{-i_0j_0})} (T) =AF_k(N,v,\Gamma^{-i_0j_0}).
\end{multline}

\end{proof}

Observe that this result implies that deleting the arcs which do not belong to the local digraph of a vertex does not affect its Average Forest measure. In other words, the AF of a vertex depends only of its local digraph.

In particular, assuming $k=j_0$ and taking into account that $\Gamma(j_0)=\Gamma^{-i_0j_0} (j_0)$, the successor equivalent property follows immediately from the previous proposition:

\begin{corollary}
The Average Forest measure verifies successor equivalence.
\label{SE-proposition}
\end{corollary}

Let us now illustrate the situation by means of an example.

\begin{example}
Using again Example \ref{short-example-forests-family}, we see that the deletion arc $(1,3)$ does not affect the Average Forest measure of player 3, whereas deleting arc $(4,2)$ does not affect the average forest measure of agents 2 and 3. In the first case the forests $F_{1,1}$ and $F_{5,1}$ disappear, but $AF_3(v,\Gamma^{-13})=v(3)=AF_3(v,\Gamma)$. In the second one, the new family ${\cal F}^{\Gamma^{-42}}$ has the same number of elements with $F_{i,j}^{-42}=F_{i,j}\setminus (4,2)$ and being $m_2^{(v,\Gamma^{-42})}(F_{i,j}^{-42})=m_2^{(v,\Gamma)}(F_{i,j})$, and $m_3^{(v,\Gamma^{-42})}(F_{i,j}^{-42})=m_3^{(v,\Gamma)}(F_{i,j})$, for all $(F_{i,j}^{-42})$.
\end{example}

The previous results show which players remain unaffected when an arc $(i_0,j_0)$ is deleted. The next results show how the remaining agents could be affected depending on their relative position with respect the agents involved in the removed arc, $i_0$ and $j_0$, and the properties of the game. Note that only the predecessors of $j_0$ can be affected by the deletion of that arc.

In the sequel, we must distinguish between {\em direct} and {\em indirect} effects of removing an arc, as well as {\em conflictive} and {\em non-conflictive} effects. Roughly speaking, direct effects refer to how removing arc $(i_0,j_0)$ affects the direct predecessors of $j_0$, whereas indirect effects refer to how these direct effects spread through the remaining predecessors of $j_0$ via direct ones. On the other hand, non-conflictive effects refer to those cases in which all the players in every reporting chain from a given predecessor $k$ to $j_0$ undergo the same kind of effect, positive or negative. For instance, in Example \ref{short-example-forests-family}, deleting the arc $(1,3)$ has a direct positive effect over player 2; a direct and an indirect (via 2) effect over player 4, being both non-conflictive effects; an indirect positive effect (via 4 and 2) over player 5; and a direct negative effect over player 1, who also undergoes a conflictive positive indirect effect (via 4).
When the effects are non-conflictive, superadditivity guarantees positive direct effects, whereas convexity must be assumed in order to guarantee indirect positive effects.

Formally, we will refer to the set of {\em direct competitors} of $i_0$ over $j_0$, which we denote by $DC(i_0,j_0)$, as the set $DC(i_0,j_0):= \hat{P}_{\Gamma} (j_0)\setminus \{ i_0\}$. We also consider the set of {\em indirect competitors} of $i_0$ over $j_0$, which we denote by $IC(i_0,j_0)$ and define as the set of predecessors of the players $k$ such that $k\in DC(i_0,j_0)$.

\begin{example}
Formally, if we go back to Example \ref{short-example-forests-family}, $DC(1,3)=\{2,4\}$ and $IC(1,3)=\{1,4,5\}$. Proposition  \ref{direct-competitors-effect} and Proposition \ref{indirect-competitors-effect}, stated next, generalise the following relations between the Average Forest measure for the original digraph $(N,\Gamma)$ and the corresponding value of the measure in the digraph $(N,\Gamma^{-13})$. To be specific:
\begin{itemize}
\item[$1.$]
$DC(1,3)\setminus IC(1,3)=\{ 2\}$. Then, direct competition reduction effect over player 2, assures $AF_2 (N,v,\Gamma)\leq AF_2 (N,v,\Gamma^{-13})$ whenever $v$ is superadditive.
\item[$2.$]
$1\in IC(1,3)$. Thus, there are conflictive effects over player 1 and relation between $AF_1 (N,v,\Gamma)$ and $AF_1 (N,v,\Gamma^{-13})$ can not be established in general.
\item[$3.$]
$\overline{P}_{\Gamma} (1)=\{ 1\}$ and $IC(1,3)\setminus \overline{P}_{\Gamma} (1)=\{ 4,5\}$.  Then, indirect non conflictive competition reduction effect over players 4 and 5, assures $AF_4 (N,v,\Gamma)\leq AF_4 (N,v,\Gamma^{-13})$ and $AF_5 (N,v,\Gamma)\leq AF_5 (N,v,\Gamma^{-13})$ whenever $v$ is convex.
\item[$4.$]
Since $1\in IC(1,3)$, and therefore the direct negative effect over player 1 can be compensated by the indirect positive effect via player 4,  indirect negative effects can not also be guaranteed.
\end{itemize}
\end{example}

\begin{proposition}
Let $(N,v,\Gamma)\in{\cal G}^N$, and let $\Gamma^{-i_0j_0}=\Gamma\setminus \{ (i_0,j_0)\}$. Then, if the game $(N,v)$ is superadditive the following holds:
\begin{enumerate}[$(i)$]
\item
Direct competition reduction effect: for every $i_k\in DC(i_0,j_0)$ with $i_k\notin P^\Gamma (j)$, for all $j\in \hat{P}^\Gamma (j_0)$, we have $AF_{i_k} (N,v,\Gamma)\leq AF_{i_k} (N,v,\Gamma^{-i_0j_0})$.

\item
Direct subordinate reduction effect: if $i_0$ is not also an indirect competitor of itself over $j_0$, i.e., $i_0\notin IC(i_0,j_0)$,  then $AF_{i_0} (N,v,\Gamma)\geq AF_{i_0} (N,v,\Gamma^{-i_0j_0})$.
\end{enumerate}
\label{direct-competitors-effect}
\end{proposition}

\begin{proof}
Let $\hat{P}^\Gamma (j_0)=\{ i_0,i_1,\dots,i_{\delta^-_{\Gamma} (j_0)-1}\}$. The proof relies on the partition of ${\cal F}^\Gamma$ given by:
$$
{\cal F}^\Gamma_k=\{ F\in {\cal F}^\Gamma \text{ s.t. } (i_k,j_0)\in F\}, \quad k=0,1,\dots, \delta^-_{\Gamma} (j_0)-1.
$$
Note that for all $i_k\neq i_0$ with $i_k\notin P^\Gamma (j)$, for all $j\in \hat{P}^\Gamma (j_0)$, the following holds:
\begin{align}
m_{i_k}(F_k) &=v(T_{i_k} (F_k))-\sum_{j\in \hat{S}^{F_k} (i_k) } v(T_j(F_k))=
\label{1}
\\[2mm]
 & =
 \displaystyle v( {\displaystyle\cup_{j\in \hat{S}^{F_k}(i_k) \setminus j_0}}  T_{j}(F_k) \cup T_{j_0} (F_k) \cup i_k) - v( T_{j_0} (F_k)) -\sum_{j\in {\hat{S}^{F_k}(i_k) \setminus j_0}} v(T_{i_k}(F_k),
 \label{2}
\end{align}
and, taking $F_0=F_k\setminus (i_k,j_0)\cup (i_0,j_0)\in {\cal F}^\Gamma_0$, it is verified $\hat{S}^{F_0} (i_k)=\hat{S}^{F_k}(i_k) \setminus j_0$ and $ T_{i_k} (F_0)=T_{i_k} (F_k) \setminus T_{j_0} (F_k)$, being $T_{j} (F_k)=T_{j} (F_0)$, for all $j\in \hat{S}^{F_k} (i_k)$. Therefore:
\begin{align}
m_{i_k}(F_0)& =v(T_{i_k} (F_0))-\sum_{j\in \hat{S}^{F_0} (i_k) } v(T_j(F_0))=
\label{3}
\\[2mm]
& =v( {\displaystyle\cup_{j\in \hat{S}^{F_k}(i_k) \setminus j_0}}  T_{i_k}(F_k) \cup i_k) - \sum_{j\in {\hat{S}^{\Gamma}(i_k) \setminus j_0}} v(T_{i_k}(F_k)).
\label{4}
\end{align}

Remark that $F_t\in {\cal F}^\Gamma_t \Leftrightarrow
F_t\setminus (i_t,j_0)\cup (i_s,j_0)\in {\cal F}^\Gamma_s$, for every $s,t$. Thus, denoting by $
{\cal F}^\Gamma_{-j_0}:= \{ F_t\setminus (i_0,j_0)\, \vert \, F_0\in {\cal F}^\Gamma_0\}$, the following is true ${\cal F}^\Gamma_t=\{ (i_t,j_0)\cup F\, \vert\, F\in {\cal F}^\Gamma_{-j_0}\}$, for  all $t$.

Analogously, a partition of ${\cal F}^{\Gamma^{-i_0j_0}}$ can be defined, with ${\cal F}^{\Gamma^{-i_0j_0}}_t={\cal F}^{\Gamma}_t=\{ (i_t,j_0)\cup F\, \vert\, F\in {\cal F}^\Gamma_{-j_0}\}$, for all $t=1,\dots, \delta^-_\Gamma(j_0)$.

Thus, $AF_{i_k}(N,v,\Gamma)$ and $AF_{i_k}(N,v,\Gamma^{-i_0j_0})$, can be computed as follows:

\begin{align}
AF_{i_k}(N,v,\Gamma)& =\sum_{t=0}^{\delta^-_{\Gamma} (j_0)-1} \sum_{F\in {\cal F}^{\Gamma}_{-j_0}}p_{\Gamma}\{\widetilde{F}=F\cup (i_t,j_0)\} m_{i_k} (F \cup (i_t,j_0)), 
\\
AF_{i_k}(N,v,\Gamma^{-i_0j_0})& =\sum_{t=1}^{\delta^-_{\Gamma} (j_0)-1} \sum_{F\in {\cal F}^{\Gamma}_{-j_0}}p_{\Gamma^{-i_0j_0}}\{\widetilde{F}=F\cup (i_t,j_0)\} m_{i_k} (F \cup (i_t,j_0)), 
\end{align}
where the involved probabilities for the formation of (uniformly) random maximal forests are given by:
\begin{align*}
p_{\Gamma}\{\widetilde{F}=F_t\}& = \prod_{i\notin So(\Gamma)}
\frac{1}{\delta^{-}_{\Gamma} (i)}=\frac{1}{\delta^{-}_{\Gamma} (j_0)}\prod_{\substack{i\notin So(\Gamma)\\ i\neq j_0}}
\frac{1}{\delta^{-}_{\Gamma} (i)},
\\
p_{\Gamma^{-i_0j_0}}\{\widetilde{F}=F_t\}& = \prod_{i\notin So(\Gamma^{-i_0j_0})}
\frac{1}{\delta^{-}_{\Gamma^{-i_0j_0}} (i)}=\frac{1}{\delta^{-}_{\Gamma} (j_0)-1}\prod_{\substack{i\notin So(\Gamma)\\ i\neq j_0}}
\frac{1}{\delta^{-}_{\Gamma} (i)}.
\end{align*}
Then, if we denote $p_F=\prod_{\substack{i \notin So(\Gamma) \\ i \neq j_0}} \frac{1}{\delta^{-}_{\Gamma} (i)}$, and take into account the following facts:
\begin{itemize}
  \item[(a)] for all $t\neq k $, if $F_0=F\cup (i_0,j_0)\in {\cal F}^\Gamma_0$ and $F_t=F\cup (i_t,j_0)\in {\cal F}^\Gamma_t$, then $m_{i_k}(F_t)=m_{i_k}(F_0)$,
    \item[(b)] from \eqref{1} to \eqref{4}, superadditivity implies $m_{i_k}(F_k)\geq m_{i_k}(F_0)$,  
\end{itemize}
the following inequality holds after easy algebraic manipulations noting (b), and thus $(i)$ is true:
\begin{multline*}
AF_{i_k}(N,v,\Gamma) =\frac{\delta^-_{\Gamma} (j_0)-1}{\delta^-_{\Gamma} (j_0)}
\sum_{\substack{F\in {\cal F}^{\Gamma}_{-j_0} \\ F_0=F\cup (i_0,j_0)}} p_F m_{i_k} (F_0) +
 \frac{1}{\delta^-_{\Gamma} (j_0)}
\sum_{\substack{F\in {\cal F}^{\Gamma}_{-j_0} \\ F_k=F\cup (i_k,j_0)}} p_F m_{i_k} (F_k) \leq 
\\[2mm]
\leq \frac{\delta^-_{\Gamma} (j_0)-2}{\delta^-_{\Gamma} (j_0)-1}
\sum_{\substack{F\in {\cal F}^{\Gamma}_{-j_0} \\ F_0=F\cup (i_0,j_0)}} p_F m_{i_k} (F_0) +
 \frac{1}{\delta^-_{\Gamma} (j_0)-1}
\sum_{\substack{F\in {\cal F}^{\Gamma}_{-j_0} \\ F_k=F\cup (i_k,j_0)}} p_F m_{i_k} (F_k) = AF_{i_k}(N,v,\Gamma^{-i_0j_0}).
\end{multline*}

In order to prove $(ii)$, the same reasoning applies taking into account that also $m_{i_0}(F_t)=m_{i_0}(F_k)$, for all $t,k\in \{1,\dots, \delta^-_{\Gamma} (j_0)-1\}$, and $m_{i_0} (F_0)\geq m_{i_0} (F_k)$, 
\begin{multline*}
AF_{i_0}(N,v,\Gamma) =\frac{\delta^-_{\Gamma} (j_0)-1}{\delta^-_{\Gamma} (j_0)}
\sum_{\substack{F\in {\cal F}^{\Gamma}_{-j_0} \\ F_k=F\cup (i_k,j_0)}} p_F m_{i_0} (F_k) +
 \frac{1}{\delta^-_{\Gamma} (j_0)}
\sum_{\substack{F\in {\cal F}^{\Gamma}_{-j_0} \\ F_0=F\cup (i_0,j_0)}} p_F m_{i_0} (F_0) \geq 
\\[2mm]
\geq \frac{\delta^-_{\Gamma} (j_0)-1}{\delta^-_{\Gamma} (j_0)-1}
\sum_{\substack{F\in {\cal F}^{\Gamma}_{-j_0} \\ F_k=F\cup (i_k,j_0)}} p_F m_{i_0} (F_k)  = AF_{i_0}(N,v,\Gamma^{-i_0j_0}).
\end{multline*}

\end{proof}

\begin{proposition}
Let $(N,v,\Gamma)\in{\cal G}^N$, and let $\Gamma^{-i_0j_0}=\Gamma\setminus \{ (i_0,j_0)\}$. Then, if the game $(N,v)$ is convex the following holds:
\begin{itemize}
\item[$(iii)$]
Indirect competition reduction effect: for every indirect competitor $k$ of $i_0$ over $j_0$ who is not also a predecessor of $i_0$, i.e. $k\in IC(i_0,j_0)\setminus \overline{P}_{\Gamma} (i_0)$, we have $AF_{k} (N,v,\Gamma)\leq AF_{k} (N,v,\Gamma^{-i_0j_0})$.
\item[$(iv)$]
Indirect subordinate reduction effect: if $i_0\notin IC(i_0,j_0)$  then $AF_{k} (N,v,\Gamma)\geq AF_{k} (N,v,\Gamma^{-i_0j_0})$, for every $k\in P_{\Gamma} (i_0)\setminus IC(i_0,j_0)$.
\end{itemize}
\label{indirect-competitors-effect}
\end{proposition}

Note that condition $(iv)$ also holds for every direct competitor of $i_0$ over $j_0$ that is not a predecessor of $i_0$, generalising condition $(i)$ to direct competitors that are also predecessors of other direct competitors.

\begin{example}
Before proving the previous proposition, let us reconsider Example \ref{short-example-forests-family}, to show that convexity is a necessary condition for ensuring that the statements $(iii)$ and $(iv)$ are true. Let $(N,v)$, given by $v(i)=0$, $\forall\, i$, $v(S)=0$, $\forall \, S\subseteq \{1,2,4\}$, $v(\{3,5\})=5$, $v(S\cup 3)=v(S\cup 5)=3 +(s+1)^2$, $v(S\cup \{3,5\})=5+(s+2)^2$,
$\forall \, \emptyset\neq S\subseteq \{1,2,4\}$.

Now, deleting arc $(1,3)$, the partitions involved in Proposition \ref{direct-competitors-effect} are given by:
\begin{align}
{\cal F}^\Gamma & = {\cal F}_1^\Gamma \cup {\cal F}_2^\Gamma \cup {\cal F}_4^\Gamma,
\\
{\cal F}^{\Gamma^{-13}} & = {\cal F}_2^\Gamma \cup {\cal F}_4^\Gamma,
\end{align}
where ${\cal F}_1^\Gamma =\{ F_{1,1}, F_{5,1}\}$, ${\cal F}_2^\Gamma =\{ F_{1,2}, F_{5,2}\}$, and ${\cal F}_4^\Gamma =\{ F_{1,4}, F_{5,4}\}$.

In this case, player 2 is a direct competitor of 1 over 3 that is not a predecessor of 1 nor 4, then superadditivity of $v$ ensures condition $(i)$ and $AF_{2} (v,\Gamma)\leq AF_{2} (v,\Gamma^{-13})$. Analogously, since
$$
MC_4(\{2\})=v(\{4,2\})-v(\{2\})= 0 < 5= v(\{4,2,3\})-v(\{2,3\})= MC_4(\{2,3\}),
$$
$AF_{4} (N,v,\Gamma) < AF_{4} (N,v,\Gamma^{-13})$. On the contrary, since 
$$MC_5(\{4,2\})=v(\{5,4,2\})-v(\{4,2\})= 12 > 9= v(\{5,4,2,3\})-v(\{4,2,3\})= MC_5(\{4,2,3\}),
$$
$AF_{5} (N,v,\Gamma) > AF_{5} (N,v,\Gamma^{-13})$. 
\label{example-convexity}
\end{example}

\begin{proof}
The proof relies on the same partitions of ${\cal F}^\Gamma$ and ${\cal F}^{\Gamma^{-i_0j_0}}$.

Let us prove $(iii)$. Let $k$ be an indirected competitor of $i_0$ over $j_0$ who is not also a predecessor of $i_0$.  Let $F_0\in {\cal F}^\Gamma_0$ and consider $F_0\setminus (i_0,j_0)\cup (i_t,j_0)\in {\cal F}^\Gamma_t$, for every $t\in \hat{P}^\Gamma (j_0)\setminus i_0$. Define $B_F(k)\subseteq \hat{P}^{\Gamma} (j_0)$ be the set of direct predecessors of $j_0$ in the original digraph $\Gamma$ such that $k\in \hat{P}^{F_t} (j_0)$, and denote as $b_F(k)$ its cardinality. Note that in this case $B_F(k)\subseteq DC(i_0,j_0)$, since $k$ is not also a predecessor of $i_0$.

Then, 
convexity and marginal contributions definition, assures that:
\begin{itemize}
  \item[(a)] for all $t\in B_F(k)$, then $m_{k}(F_t) \geq m_{k}(F_0)$,
    \item[(b)]  for all $t\in \hat{P}^\Gamma(j_0)\setminus B_F(k)$, then $m_{k}(F_t) = m_{k}(F_0)$.  
\end{itemize}

Thus, for each $F_0\in {\cal F}^\Gamma$ and $F_t$, $t\in  \hat{P}^\Gamma(j_0)\setminus i_0$, 
the following inequality holds after easy algebraic manipulations noting (a) and (b), and thus $(iii)$ is true:
\begin{multline*}
\sum_{s=0}^{\delta^-_{\Gamma} (j_0)-1} p_\Gamma\{\widetilde{F}=F_s\} m_k(F_s)=\frac{\delta^-_{\Gamma} (j_0)-b_F(k)}{\delta^-_{\Gamma} (j_0)} p_F m_{k} (F_0) +
 \frac{1}{\delta^-_{\Gamma} (j_0)}
p_F \sum_{t\in B_F(k)}  m_{k} (F_t) \leq 
\\[2mm]
\leq \frac{\delta^-_{\Gamma} (j_0)-b_F(k)-1}{\delta^-_{\Gamma} (j_0)-1}
p_F m_{k} (F_0) +
 \frac{1}{\delta^-_{\Gamma} (j_0)-1} p_F
\sum_{t\in B_F(k)}  m_{k} (F_t)=\sum_{s=1}^{\delta^-_{\Gamma} (j_0)-1} p_{\Gamma^{-i_0j_0}} \{\widetilde{F}=F_s\} m_k(F_s),
\end{multline*}
where $p_F=\prod_{\substack{i \notin So(\Gamma) \\ i \neq j_0}} \frac{1}{\delta^{-}_{\Gamma} (i)}$.

In order to prove $(iv)$ note that if $i_0\notin IC(i_0,j_0)$, and $k\in P_{\Gamma} (i_0)\setminus IC(i_0,j_0)$, then, for any given collection of maximal spanning forest $\{F_s, s=0,\dots, \delta^-_\Gamma (j_0)-1\}$, $B_F(k)=\{ i_0\}$. Therefore, taking into account that convexity implies $m_k(F_0)\geq m_k(F_t)$, for all $t\neq 0$:

\begin{multline*}
\sum_{s=0}^{\delta^-_{\Gamma} (j_0)-1} p_\Gamma\{\widetilde{F}=F_s\} m_k(F_s)=\frac{1}{\delta^-_{\Gamma} (j_0)} p_F m_{k} (F_0) +
 \frac{1}{\delta^-_{\Gamma} (j_0)}
p_F \sum_{t=1}^{\delta^-_\Gamma(j_0)-1}  m_{k} (F_t) \geq 
\\[2mm]
\geq \frac{1}{\delta^-_{\Gamma} (j_0)-1}
\sum_{t=1}^{\delta^-_\Gamma(j_0)-1}  m_{k} (F_t)=\sum_{s=1}^{\delta^-_{\Gamma} (j_0)-1} p_{\Gamma^{-i_0j_0}} \{\widetilde{F}=F_s\} m_k(F_s).
\end{multline*}

\end{proof}

The predecessors of $j_0$ not considered in the previous propositions undergo a conflictive effect. In fact, some of their subordinates lie in cases $(i)$ or $(iii)$, whereas some others lie in $(ii)$ or $(iv)$. Thus, the deletion of arc $(i_0,j_0)$ causes a positive, but also a negative, effect over them. The balance between those conflictive effects depends crucially of the specific values of the measures of the involved coalitions. We must remark that if the underlying digraph $\Gamma=(N,A)$ is cycle-free then there are no conflictive effects and Propositions \ref{direct-competitors-effect} and \ref{indirect-competitors-effect} describe the effect over all predecessors of $j_0$.










\section{Computation of the Average Forest measure}

Enumerating all spanning trees/forests of an acyclic digraph is computationally intensive (can be exponential in number), but not NP-hard. Therefore, as proposed for other cooperative games values (Castro et al., 2009), we can estimate the $AF$ measure by sampling methods.

First, it is necessary to randomly generate a sample of forests; then, to compute the vector of marginal contributions for each of the sampled forests and, finally, to compute the mean of the obtained marginal contributions for each node.

For the generation of the sample of forests, $S({\cal F}^{\Gamma})$, we can follow the constructive method introduced in section \ref{delarc} and described in Algorithm \ref{alg:F_sampling}. Then, we can use common techniques of simple sampling with replacement for obtaining the sample size $k$ required for a given precision.

\begin{algorithm}
\caption{[$\Gamma=(N,A)$] = $F\in {\cal F}^\Gamma$ random generation }
\label{alg:F_sampling}
\begin{algorithmic}[1]
\REQUIRE Given $N$ and $A$.
  \STATE Initialize $F=(N,A_F)$ with all arcs $(i,j)\in A$ with $\delta^-_\Gamma(j)=1$
   \FORALL{$j\in N$ s.t. $\delta^-_\Gamma(j)>1$:}
     \STATE Select randomly a direct predecessor $i\in \hat{P}^\Gamma (j)$
     \STATE $A_F\leftarrow A_F\cup (i,j)$   
   \ENDFOR
   \end{algorithmic}
   \end{algorithm}


To evaluate marginal contributions, we propose the Algorithm \ref{alg:mc_calculus}. 
\begin{algorithm}
\caption{[$(N,v,F)$] = $m^{(v,\Gamma)}_F$ computation}
\label{alg:mc_calculus}
\begin{algorithmic}[1]
\REQUIRE Given $F\in S({\cal F}^{\Gamma})$, $N$.
  \STATE {$F'\leftarrow F$}
   \WHILE{no nodes remain in $N$}
    \IF{ $\delta_{F}^+(i)=0$}
    \STATE $\widehat S^F(i)$, $\bar S^F(i)$ and
\begin{equation*}
m_i^{(v,\Gamma)}(F)=v(\bar S^F(i))-
       \sum\limits_{j\in\widehat S^F(i)}v(\bar S^F(j)),
               \qquad\mbox{for all}\enspace i\in N,
\end{equation*} 
   \STATE Eliminate the set of nodes with $\delta_{F'}^+(i)=0$ denoted by $D_{F'}^0$  from the forest $F'$ and from $N$: $F'\leftarrow F'\setminus D_{F'}^0$,  $N \leftarrow N\setminus D_{F'}^0$.
   \ENDIF
   \ENDWHILE
   \end{algorithmic}
   \end{algorithm}

Now, applying  Algorithm \ref{alg:mc_calculus} to every $F\in S({\cal F}^{\Gamma})$, calculate an estimation of the Average Forest measure as follows: 
$$
\hat{AF}(N,v,\Gamma)=\frac{1}{|S({\cal F)}^\Gamma|}
     \sum\limits_{F\in S({\cal F}^{\Gamma})}m^{(v,\Gamma)}(F).
$$

To sum up, Algorithm \ref{alg:AF_estimation} is proposed to estimate the average forest measure of a given organisational situation $(N,v,\Gamma)$  by means of Monte Carlo simulation. 

If $\vert {\cal F}^\Gamma\vert << k$, where $k$ is the required sample size, then Algorithm \ref{alg_AFV_estimation} can be used to calculate the exact average forest measure of the organisational situation. This is achieved by listing all maximal forests in the set  ${\cal F}^\Gamma$, instead of generating a sample of maximal forest $S({\cal F}^{\Gamma})$.

\begin{algorithm}
\caption{[$(N,v,\Gamma,k)$] = $AF$\_estimation($N,v,A,k$)}
\label{alg:AF_estimation}
\begin{algorithmic}[1]
\REQUIRE organisational situation $(N, v,\Gamma)$, sample size $k$
\ENSURE {$\hat{AF} (N,v,\Gamma)$}
 \STATE Initialize $\hat{AF} (N,v,\Gamma)\leftarrow 0$,  for all $i \in N$
\FOR{$r = 1$ to $k$}
   \FORALL{$i\in N$}
       \STATE Initialize $Team\_group(i) \leftarrow \{i\}$
       \STATE Initialize $team\_value(i) \leftarrow v(i)$
       \STATE Initialize $Dep\_teams(i) \leftarrow \emptyset$
       \STATE Initialize $num\_dependent\_teams(i) \leftarrow 0$
  \ENDFOR   
    \STATE Initialize $F=(N,A_F)$ with $A_F\leftarrow \emptyset$
    \FORALL{$j\in N$ s.t. $\delta^-_\Gamma(j)=1$:}
      \STATE $A_F\leftarrow A_F\cup (i,j)$, with $\{i\}=\hat{P}^\Gamma (j)$
      \STATE $Dep\_teams(i) \leftarrow Dep\_teams(i) \cup \{ j\}$
      \STATE $num\_dependent\_teams(i) \leftarrow num\_dependent\_teams(i) +1$
    \ENDFOR  
   \FORALL{$j\in N$ s.t. $\delta^-_\Gamma(j)>1$:}
     \STATE Select randomly a direct predecessor $i\in \hat{P}^\Gamma (j)$
     \STATE $A_F\leftarrow A_F\cup (i,j)$  
     \STATE $Dep\_teams(i) \leftarrow Dep\_teams(i) \cup \{ j\}$
      \STATE $num\_dependent\_teams(i) \leftarrow num\_dependent\_teams(i) +1$     
   \ENDFOR   
   \STATE $non\_evaluated \leftarrow n$ 
   \WHILE{$non\_evaluated >0$}
      \FORALL{$i\in N$ s.t. $num\_dependent\_teams(i) =0$:}
       \STATE $Team\_group(i) \leftarrow \displaystyle \bigcup_{j\in Dep\_teams(i)} Team\_group(j) \cup Team\_group(i)$ 
       \STATE $team\_value(i) \leftarrow v(Team\_group(i))$
       \STATE $\hat{AF}_i (N,v,\Gamma) \leftarrow \hat{AF}_i (N,v,\Gamma) + \displaystyle \frac{1}{k} \Bigl ( team\_value(i)-\displaystyle \sum_{j\in Dep\_teams(i)}  team\_value(j)\Bigr )$
       \STATE $non\_evaluated \leftarrow non\_evaluated -1 $
    \ENDFOR
  \ENDWHILE
\ENDFOR
\RETURN $\hat{AF}_i (N,v,\Gamma)$, $i\in N$
\end{algorithmic}
\label{alg_AFV_estimation}
\end{algorithm}




    

%

\section{Conclusions}

This paper introduces a novel framework for evaluating leadership and productivity within organisations modeled as directed graphs (digraphs). Through the concept of the \textbf{Average Forest measure (AF)}, we propose a method to assess each individual's potential for leading productive working teams are represented as arborescences within maximal spanning forests of the digraph. Under the assumption that all such configurations are equally probable, the AF leadership measure captures the expected marginal contribution of a player across all team arrangements.

The approach builds on transferable utility (TU) game theory and assumes superadditivity to ensure that team worth increases when coalitions merge. 

We have studied some properties of the measure as nonnegativity, linearity, dummy player and monotonicity. It is remarkable that an specific concept of the dummy property must be considered in this context, to consider not only the capacity of the agent to add value, but his possibility of leading synergistic working teams.

The other important property is the monotonicity of the AF measure with respect to the worth function of teams. We have considered two classical properties and introduced a new one, individual monotonicity, that guarantees the increase of the marginal contributions.

Notably, it is shown that even seemingly inessential arcs can alter leadership scores, emphasizing the sensitivity of outcomes to the specific reporting chains. It is also identified when the changes in the reporting relationships lead to positive, negative or ambiguous effects depending on structural position and also of the superadditivity and convexity of the underlying game.

We further explore how organisational structure impacts overall productivity, which we consider to be the expected value of working teams across all maximal team arrangements. Under this assumption, we have found that having a quasi-strongly connected reporting structure is a necessary and sufficient condition for maximising organisational productivity.

Lastly, due to computational complexity, a sampling-based algorithm is suggested to approximate the AF measure, making the methodology practical for real-world applications.

\end{document}